\documentclass[12pt,a4paper, reqno]{article}

\usepackage{indentfirst} 
\usepackage{amsmath, amsthm, amssymb, xfrac, mathrsfs} 
\usepackage{tikz} 
\usetikzlibrary{hobby}
\usetikzlibrary{decorations.pathreplacing}
\usepackage{wrapfig} 
\usepackage{xcolor} 
\usepackage{verbatim} 
\usepackage{enumerate} 
\usepackage{fullpage}
\usepackage{caption}
\usepackage{subcaption}

\usepackage{hyperref} 
\hypersetup{
colorlinks=true,
citecolor=black!50!red,
linkcolor=black!50!red,
linktoc=all
}
\usepackage[all]{hypcap} 

\newcounter{constant}
\newcommand{\newconstant}[1]{\refstepcounter{constant}\label{#1}}
\newcommand{\useconstant}[1]{c_{\textnormal{\tiny \ref{#1}}}}
\newcounter{bigconstant}

\newtheorem{teo}{Theorem}[section]

\newtheorem{lemma}[teo]{Lemma}

\numberwithin{equation}{section} 

\theoremstyle{definition}
\newtheorem{defn}[teo]{Definition}

\newtheorem{remark}[teo]{Remark}

\renewcommand{\P}{\mathbb{P}}

\newcommand{\R}{\mathbb{R}}

\newcommand{\N}{\mathbb{N}}
\newcommand{\Z}{\mathbb{Z}}

\newcommand{\charf}[1]{\mathbf{1}_{#1}}

\DeclareMathOperator{\poisson}{Poisson}

\title{Local and global survival for infections with recovery}

\author{Rangel Baldasso\footnote{E-mail: \ r.baldasso@math.leidenuniv.nl; \ Mathematical Institute, Leiden University, P.O. Box 9512, 2300 RA Leiden, The Netherlands.} 
\and Alexandre Stauffer\footnote{E-mail: \ astauffer@mat.uniroma3.it; \ Universit\`a Roma Tre, Dip.\ di Matematica e Fisica, Largo S.\ Murialdo 1, 00146, Rome, Italy; University of Bath, Dept of Mathematical Sciences, BA2 7AY Bath, UK, supported by EPSRC Fellowship EP/N004566/1.}}

\begin{document}

\maketitle

\begin{abstract}
  We establish two open problems from Kesten and Sidoravicius~\cite{ks}.
  Particles are initially placed on $\Z^{d}$ with a given density and evolve as independent continuous-time random walks.
  Particles initially placed at the origin are declared as infected. 
  Infection transmits instantaneously to healthy particles on the same site and infected particles become healthy with a positive rate.
  We prove that, for small enough recovery rates, the infection process survives and visits the origin infinitely many times on the event of survival.
  Second, we establish the existence of density parameters for which the infection survives for all choices of the recovery rate.
\end{abstract}

\section{Introduction}
In this note, we consider infection processes with recovery that evolve on the $d$-dimensional lattice where each individual particle performs an independent continuous-time simple random walk. 
Our goal is to settle two open problems by Kesten and Sidoravicius~\cite{ks}.

We start by defining the model precisely. 
Fix $\rho>0$ and let $(\eta_{0}(x))_{x \in \Z^{d}}$ be an i.i.d.\ collection of random variables with distribution $\poisson(\rho)$. Place, at time zero, $\eta_{0}(x)$ particles at the site $x \in \Z^{d}$ and let these particles evolve as independent continuous-time nearest-neighbor random walks. We call $\rho$ the density of the process and denote by $\eta_{t}(x)$ the number of particles at site $x$ at time $t$.

Now place at time zero an additional, infected particle at the origin (this is so that the process in nontrivial) 
and declare all other particles at the origin as infected. 
Particles starting outside the origin are declared healthy. 
Healthy particles become immediately infected when they share a site with an infected particle. 
Furthermore, infected particles become healthy with rate $\lambda>0$. 
The parameter $\lambda$ is called the recovery rate. 
Let $\P_{\rho}^{\lambda}$ denote the distribution of the process when the initial density is $\rho$ and the recovery rate is $\lambda$.
Due to the instantaneous infection mechanism, an infected particle can only become healthy if it is the only particle in the site. 

\begin{remark}\label{remark:difference_ks}
   In the model considered in~\cite{ks}, infection does not happen instantaneously. Instead, it occurs only when there is a jump from an infected particle to a site with healthy particles or, the other way around, 
   when a healthy particle jumps into a site with infected particles. In this case, the main change is that sites can have both infected and healthy particles at the same time.
\end{remark}

We say that the \emph{infection survives} if, for every positive time $t$, there exist at least one infected particle. 
We say that the \emph{infection survives locally} if the set of times the origin contains an infected particle is unbounded. 
Finally, we say that the \emph{infection dies out} if there exists a positive (random) time $t$ such that all particles are healthy at time $t$.

We study survivability and local survivability of this infection processes in the perturbative regimes, where one of the parameters $\rho$ or $\lambda$ is fixed and the other is taken either large or small enough. In~\cite{ks}, the authors prove that, for any given $\rho>0$, the infection survives with positive probability if $\lambda$ is taken small enough.

According to~\cite[Remark 2]{ks}, their proof of survival for small recovery rates does not give any information about the behavior of the infected particles. Our first result goes in this direction: we prove that, provided $\lambda$ is taken small enough, local survival occurs almost surely, conditioned on survival of the infection.

\begin{teo}\label{t:local_survival}
   For any density $\rho>0$, if the recovery rate $\lambda>0$ is small enough then 
   \begin{equation}\label{eq:local_survival}
   \P_{\rho}^{\lambda}\Big[\text{ the infection survives locally } \Big| \text{the infection survives } \Big] = 1.
   \end{equation}
\end{teo}

\begin{remark}
   The theorem above is stated for the case of instantaneous infection, but the proof also works for the case when infection happens only at jump times, as defined in Remark~\ref{remark:difference_ks}.
\end{remark}

\bigskip

The second result of this paper regards~\cite[Page 548]{ks} and is related to the absence of a phase transition for large values of the density parameter $\rho$. We prove that, provided $\rho$ is large enough, the infection survives with positive probability for all recovery rates $\lambda$.

%
%
\begin{teo}\label{t:survival}
   There exists a density $\rho_{+}>0$ such that, for all $\rho \geq \rho_{+}$ and all $\lambda\in(0,\infty]$,
   \begin{equation}
      \P_{\rho}^{\lambda}\Big[ \text{ the infection survives } \Big]>0.
   \end{equation}
\end{teo}

\begin{remark}
   The theorem above does not hold from the model considered in~\cite{ks}. Theorem~\ref{t:survival} relies on the fact that particles can only heal when they are alone in a site, which happens rarely when the density is large enough. If on the other hand we consider the case when sites can have both infected and healthy particles, the conclusions of the theorem above do not carry over, since it might be that, provided the healing rate is large enough, particles heal fast enough as to prevent the survival of the infection.
\end{remark}

As proved in Section~\ref{subsec:model}, the collection of infected particles in the process with $\lambda=\infty$ (that is, infected particles recover instantaneously whenever they are alone at a site) 
is stochastically dominated by the set of infected particles in the process with $\lambda<\infty$. 
Using this, we will restrict our proof of Theorem~\ref{t:survival} to the case $\lambda=\infty$.

\begin{remark}
   The absence of a phase transition for large densities, as shown in Theorem~\ref{t:survival}, does not occur 
   in the model when infection spreads only when particles jump, as defined in Remark~\ref{remark:difference_ks}. For that model,~\cite{ks} in fact shows that a phase transition on $\lambda$ occurs for any $\rho$. 
\end{remark}

\bigskip

\noindent\textbf{Overview of the proofs.} 
The proofs of Theorems~\ref{t:local_survival} and~\ref{t:survival} are based on a multi-scale framework developed by Gracar and Stauffer in~\cite{gs2}, and provide a novel application of this technique. 
Given a local monotone event, the results in~\cite{gs2} provide the existence of a \emph{Lipschitz surface} (see Definition~\ref{def:lipschitz_surface}) formed by boxes where translations of this event occur.
To show that the infection survives locally and establish Theorem~\ref{t:local_survival}, 
we consider a local event which guarantees that if the infection reaches a space-time box where this event holds, then the infection not only survives for a long time but 
also spreads to nearby boxes. A (by now) standard way of establishing survival of the infection is to show that there exists a ``directed percolating structure'' of such events. 
To show \emph{local} survival of the infection, this is not enough. For this, we use the geometrical structure provided by the Lipschitz surface, 
in particular the fact that the Lipschitz surface has infinitely many cells (in space-time) close to the origin of $\mathbb{Z}^d$,
which gives that by ``navigating through the surface'' the infection visits the origin infinitely often.

We remark that, for Theorem~\ref{t:survival}, our proof also provides some information about the position of infected particles. 
For example, we obtain that, on the event of survival, the infection travels with positive speed. However, due to technical limitations of the Lipschitz surface approach, this would be obtained only
for $d \geq 2$. To obtain a similar statement for the unidimensional case, a possible approach would be to use the renormalization developed by Baldasso and Teixeira in~\cite{bt2}.
But as positive speed is not the main purpose of this work (and there are other proofs of positive speed~\cite{ks,ks2,ks3,gs1}), we refrain from pursuing this direction.

\bigskip

\noindent\textbf{Related works.} The work~\cite{ks} is part of a collection of papers by the same authors that includes~\cite{ks2, ks3}. Only~\cite{ks} considers the case of infection with recovery.~\cite{ks2} provides lower 
and upper bounds for the speed of infection spread, while~\cite{ks3} strengthens these bounds to a full shape theorem. 
Following the same line, Baldasso and Stauffer~\cite{bs} consider infection processes spreading on top of \emph{biased} random walks, 
and prove the existence of a phase transition for local survival of the infection. 

When particles do not move independently, Baldasso and Teixeira~\cite{bt} consider the case where particles perform a one-dimensional zero-range process. They provide, under weak conditions on the jump rates, lower and upper bounds for the speed of the infection front. Infection processes on top of the one-dimensional exclusion process were considered by Jara, Moreno, and Ram\'{i}rez~\cite{jmr}, where the authors use regeneration arguments to conclude a law of large numbers and central limit theorem for the infection front.

Gracar and Stauffer develop the Lipschitz surface approach in~\cite{gs2, gs1}. The existence of a Lipschitz surface (see Theorem~\ref{t:lipschitz} in the current paper) was proved in~\cite{gs2}, while~\cite{gs1} used this structure to study infection spread on top of the random conductance model. They consider both infections with and without recovery. In the case without recovery, the infection travels with positive speed in any direction. When a recovery mechanism is introduced, they prove that the infection survives with positive probability, provided the recovery rate is small enough.
We remark however that, even though the results in~\cite{gs2} are stated only for dimensions $d \geq 2$, they still remain valid for the unidimensional case.

%
%

\section{Preliminaries}\label{sec:preliminaries}
This section contains the precise definition of the model and also the summary of the main results we will need from~\cite{gs2}. 
We split our discussion in two subsections. 
In the first, we define the infection process we consider and present some useful constructions of it, while the latter contains the statement of the theorem regarding the existence of Lipschitz surfaces.

\subsection{The model}\label{subsec:model}
For each $x \in \Z^{d}$ and $n \in \N$, let $(S^{x,n}_{t})_{t \geq 0}$ be an independent continuous-time simple random walk on $\Z^{d}$ with $S^{x,n}_{0}=x$. Given an initial condition $\eta_{0}:\Z^{d} \to \Z_{+}$, define the process $(\eta_{t})_{t \geq 0}$ by setting, for each time $t \geq 0$ and site $x \in \Z^{d}$,
\begin{equation}
   \eta_{t}(x) = \sum\nolimits_{y \in \Z^{d}} \sum\nolimits_{n \leq \eta_{0}(y)} \charf{\{S^{y,n}_{t}=x\}}.
\end{equation}

For each $\rho>0$, let $\P_{\rho}$ denote the distribution of the process above when $\eta_{0}$ is composed of i.i.d.\ $\poisson(\rho)$ random variables which are independent of the collection $(S^{x,n})_{x \in \Z^{d}, n \in \N}$. For each $\rho>0$, this distribution is invariant for the process $(\eta_{t})_{t \geq 0}$ and the parameter $\rho$ is called the density of the process.

Whenever we have two initial conditions $\eta_{0}$ and $\tilde{\eta}_{0}$ such that $\eta_{0} \preceq \tilde{\eta}_{0}$\footnote{Given two configurations $\eta$ and $\tilde{\eta}$, we say that $\eta \preceq \tilde{\eta}$ if $\eta(x) \leq \tilde{\eta}(x)$, for all $x \in \Z^{d}$.}, we can construct a coupling between the processes $(\eta_{t})_{t \geq 0}$ and $(\tilde{\eta}_{t})_{t \geq 0}$ in a way that $\eta_{t} \preceq \tilde{\eta}_{t}$, for all $t \geq 0$: one simply uses the same collection of paths $(S^{x,n})_{x \in \Z^{d}, n \in \N}$ to construct the evolution. In particular, it is possible to couple processes $(\eta^{\rho}_{t})_{t \geq 0}$ with different densities in a way that $\eta_{t}^{\rho} \preceq \eta_{t}^{\rho'}$, for all $t \geq 0$, if $\rho \leq \rho'$.

\bigskip

\noindent\textbf{The process with positive recovery rate.} Let us now define the infection process with recovery rate $\lambda>0$. 
To this end, we introduce an independent collection of Poisson point processes $(R^{x,n}_{\lambda})_{x \in \Z^{d}, n \in \N}$ 
on $\R_{+}$ with intensity $\lambda$. The process $R^{x,n}_{\lambda}$ is thought of as the recovery marks of the walk $(S^{x,n}_{t})_{t \geq 0}$.

At time zero, place an additional particle at the origin, declare all particles at the origin as infected, and the remaining particles as healthy. Healthy particles get infected when they share a site with an already infected particle. An infected particle $(y,n)$ recovers at time $t$ if $t \in R^{y,n}_{\lambda}$ 
and the particle is alone at that time (that is, $\eta_t(S_t^{y,n})=1$).

Let $\xi_{t}(x)$ denote the number of infected particles at site $x$ at time $t$, and notice that $\xi_{t}(x) \in \{0, \eta_{t}(x)\}$, i.e., for all times $t \geq 0$ and sites $x \in \Z^{d}$, either all particles in $x$ are healthy or all are infected.

As already noted in the introduction, we denote by $\P_{\rho}^{\lambda}$ the distribution of the process $(\eta_{t})_{t \geq 0}$ with initial density $\rho$ together with the recovery processes $(R^{x,n}_{\lambda})_{x \in \Z^{d}, n \in \N}$.

\bigskip

\noindent\textbf{Instantaneous recovery.} Finally, we define the process with instantaneous recovery, whose distribution we denote by $\P_{\rho}^{
\infty}$. The infection mechanism is the same as in the previous case. However, here an infected particle is immediately declared healthy as soon as it is alone in a site.

Once again, by using the same collections of paths for the random walks, this provides a coupling between infection processes with recovery rate $\lambda>0$ and with instantaneous recovery in a way that the number of infected particles in the 
latter is always smaller than in the former. 

\subsection{The Lipschitz surface}
We recall the Lipschitz surface structure introduced in~\cite{gs2}. In order to do so, we briefly introduce some notation from~\cite{gs2} and provide the necessary adaptations to our setting.

\bigskip

\noindent\textbf{Tessellation.} We first tessellate the space $\Z^{d}$ into boxes of side length $\ell$ and time into intervals of length $\beta$, where $\ell$ and $\beta$ are positive integers 
that will be chosen appropriately later on. 
For each $i \in \Z^{d}$ and $\tau \in \Z_{+}$, consider the box $i$ as $i \ell + [0, \ell]^{d}$ and the interval $\tau$ as $[\tau \beta, (\tau+1) \beta]$. 
We refer to the space-time box $\Big( i \ell + [0, \ell]^{d} \Big) \times [\tau \beta, (\tau+1) \beta]$ by the \emph{cell} $(i, \tau)$.

Second, we introduce three types of overlapping boxes. 
For a positive integer $\eta>0$, we consider 
the \emph{super box $i$} as $i \ell +[-\eta \ell, (\eta+1)\ell]^{d}$ and the \emph{super interval $\tau$} as $[\tau \beta, (\tau+\eta) \beta]$.
A \emph{super cell} $(i, \tau)$ is the set $\Big( i \ell +[-\eta \ell, (\eta+1)\ell]^{d} \Big) \times [\tau \beta, (\tau+\eta) \beta]$.
Finally, we define the \emph{extended box} $i$ as $i \ell + \big[-\frac{\ell}{3}, \ell+\frac{\ell}{3} \big]^{d}$ and the \emph{quasi-super box} $i$ as $i \ell + \big[-\frac{\eta \ell}{2}, \ell+\frac{\eta \ell}{2} \big]^{d}$.

For a set $X' \subset \Z^{d}$, we say that a particle $(S^{x,n}_{t})_{t \geq 0}$ has \emph{displacement} in $X'$ in the time interval $[t_{0}, t_{0}+t_{1}]$ if $S^{x,n}_{t}-S^{x,n}_{t_{0}} \in X'$, for all $t \in [t_{0},t_{0}+t_{1}]$.
An event $A$ is restricted to the space-time box $B = B_{s} \times [t_{0}, t_{0}+t] \subset \Z^{d} \times \R$ if it depends only on the particles that are in $B_{s}$ at time $t_{0}$ and their trajectories during the time interval $[t_{0}, t_{0}+t]$.

The space of space-time configurations can be endowed with a partial order via $\xi \preceq \tilde{\xi}$ if $\xi_{t}(x) \leq \tilde{\xi}_{t}(x)$, for all $(x,t) \in \Z^{d} \times \R_{+}$. An event $A$ is said \emph{increasing} if $\xi \preceq \tilde{\xi}$ implies $\textbf{1}_{A}(\xi) \leq \textbf{1}_{A}(\tilde{\xi})$.

\begin{defn}
   For a density $\rho>0$ and an increasing event $E$ restricted to the space-time box $X \times [0,s]$, the \emph{probability associated} to $E$ is the probability $\nu_{E}(\rho, X, X', s)$ that 
   $E$ holds given that, at time $0$, the particles in $X$ have density $\rho$ (meaning that each site in $X$ has a number of particles that is independently distributed according to a $\poisson(\rho)$ distribution) and displacement in $X'$ from time $0$ to $s$.
\end{defn}

Let $E$ denote an increasing event that is restricted to the super cell $(0,0)$. For each $(i,\tau) \in \Z^{d+1}$, let $E(i,\tau)$ denote the translation of the event $E$ to the super cell $(i, \tau)$.

\bigskip

\noindent\textbf{The base-height index.} The base-height index is a different parametrization of space-time boxes. We fix one of the $d$ spatial dimensions and denote this dimension as the \emph{height}, by using the special notation $h$ for it. The other $d-1$ spatial dimensions and the temporal dimension form the \emph{base} of our new coordinate system, and are jointly denoted by $b \in \Z^{d}$.

This new notation provides a bijection $\varphi$ between space-time coordinates $(i, \tau)$ and base-height indices $(b,h)$. The \emph{base-height cell} $(b,h)$ corresponds to the cell $(i,\tau) = \varphi^{-1}(b,h)$. Analogous definitions hold for the \emph{base-height super cell}. Furthermore, if $E(i, \tau)$ denotes an event supported in the super cell $(i, \tau)$, we write $E_{\textnormal{bh}}(b,h)$ for the analogous event in base-height notation.

\bigskip

\noindent\textbf{Two-sided Lipschitz surfaces.} A function $F:\Z^{d} \to \Z$ is Lipschitz if
\begin{equation}
|F(x)-F(y)| \leq ||x-y||_{1}, \text{ for all } x, y \in \Z^{d}.
\end{equation}

\begin{defn}\label{def:lipschitz_surface}
   A \emph{two-sided Lipschitz surface} $\mathcal{S}$ is a collection of base-height cells
   \begin{equation}
      \mathcal{S} = \{(b, F_{+}(b)), (b,-F_{-}(b)): b \in \Z^{d}\},
   \end{equation}
   where $F_{+}, F_{-}: \Z^{d} \to \Z_+$ are non-negative Lipschitz functions.
\end{defn}

\newconstant{c:ls1}
\newconstant{c:ls2}
For $r \geq 0$, denote by $Q_{r}$ the $d$-dimensional cube $[-r/2, r/2]^{d}$. The following theorem is the main result of~\cite{gs2}.
\begin{teo}[The Lipschitz surface]\label{t:lipschitz}
   There exist positive constants $\useconstant{c:ls1}$ and $\useconstant{c:ls2}$ such that the following holds. Tessellate $\Z^{d}$ in space-time cells and super cells as described above for some positive $\ell, \beta$, and $\eta$ such that $\beta/\ell^{2}$ is small enough. Let $E(i, \tau)$ be an increasing event that is restricted to the space-time super cell $(i, \tau)$. Fix $\epsilon \in (0,1)$ and $\omega$ such that
   \begin{equation}\label{eq:condition_omega}
   \omega \geq \sqrt{\frac{\eta \beta}{\useconstant{c:ls2}\ell^{2}} \log \Big(\frac{8\useconstant{c:ls1}}{\epsilon}\Big)}.
   \end{equation}
   Then there exists a positive number $\alpha_{0}$ that depends on $\epsilon$, $\eta$ and the ratio $\beta/\ell^{2}$ so that, if
   \begin{equation}
   \min\Big\{\epsilon \rho \ell^{d},\log\Big(\frac{1}{1-\nu_{E}((1-\epsilon)\rho, Q_{(2\eta+1)\ell}, Q_{\omega \ell}, \beta)} \Big) \Big\} \geq \alpha_{0},
   \end{equation}
   a two-sided Lipschitz surface $\mathcal{S}$ where $E(i,\tau)$ holds for all $(i,\tau) \in \mathcal{S}$ exists and surrounds the origin at a finite distance almost surely.
   In addition, when $d \geq 2$, if $\ell$ and $\P[E]$ are sufficiently large, then the portion of the Lipschitz surface with $h=0$ percolates in $\Z^{d}$.
\end{teo} 

\begin{remark}
The theorem above is stated in~\cite{gs2} only for dimensions $d \geq 2$, as their goal was to study high-dimensional models. Their proof, however, also applies for $d=1$. Central to this argument is~\cite[Proposition 6.1]{gs2}, which also holds in dimension $d=1$. In particular, the decay on the probability stated in that proposition follows from~\cite[Lemma 5.6]{gs2}, which is stated in all dimensions.
\end{remark}

\section{Local survival}\label{sec:local_survival}
We now prove Theorem~\ref{t:local_survival}; that is, we show that, 
for any given density $\rho>0$, provided the recovery rate $\lambda>0$ is small enough, 
the origin is visited infinitely often by infected particles on the event where the infection survives.
For this, we will use the Lipschitz Surface argument described in the previous section. 

\subsection{The Lipschitz surface and infections with recovery}
Our first focus is on applying the Lipschitz surface for the infection with recovery. 
We will consider genealogical paths of infection composed of particles that remain infected for a long time, and will introduce an event to apply Theorem~\ref{t:lipschitz}.
We do this in two steps. First, we say that the cell $(i,\tau)$ is \emph{acceptable} if the following two conditions hold on the super cell $(i,\tau)$:
\begin{enumerate}
   \item for every $x$ in the box $i$ such that $\eta_{\tau \beta}(x)>0$, there exists a path (that we call $\gamma^{x}$) that starts at $x$ and that, up to time $T=\tau\beta+\ell^{\frac{5}{3}}$, has no recovery marks and does not exit the extended box $i$;
   
   \item for each $i' \in \Z^{d}$ with $||i'||_{\infty} \leq d+2$, and each $x$ in the box $i$ such that $\eta_{\tau \beta}(x)>0$, the following holds.
      There exists at least one particle that is inside the quasi-super box $i$ at time $\tau \beta$, does not have any recovery marks and is contained in the super cell $(i,\tau)$ up to time $(\tau+1)\beta$, intersects\footnote{We say that two paths intersect if they are at the same site for a positive time interval.} the path $\gamma^{x}$ of the particle that started at $x$ before time $T$, and is, at time $(\tau+1)\beta$ inside the cell $(i+i', \tau+1)$. See Figure~\ref{fig:acceptable}~(b).
\end{enumerate}

See Figure~\ref{fig:acceptable} for a representation of an acceptable cell.

\begin{figure}[h]
\centering
\begin{subfigure}{0.4\textwidth}
\centering
\begin{tikzpicture}[scale=0.8]

\draw[<->](-3, 0) -- (4, 0);
\draw[->](-2.5, 0) -- (-2.5, 3);
\draw[dotted] (-2.5, 2.5) -- (2, 2.5);

\draw (-1, -0.05) -- (-1, 0.05);
\draw (1.5, -0.05) -- (1.5, 0.05);
\draw (2.5, -0.05) -- (2.5, 0.05);
\draw (-2, -0.05) -- (-2, 0.05);
\draw (-2.55, 2.5) -- (-2.45, 2.5);
\draw[line width=2pt, blue!40!gray] (-1,0) -- (1.5,0);

\node[left] at (-2.5 , 2.5) {\small{$T$}};

\node[below] at (-2, 0) {\small{$i\ell-\frac{\ell}{3}$}};
\node[below] at (2.5, 0) {\small{$(i+1)\ell+\frac{\ell}{3}$}};

\draw (0,0) to [out=90, in=200] (1,1) to [out=20, in=270] (2, 2.5);
\node[below] at (1,1) {\small{$\gamma^{x}$}};

\draw (0, -0.05) -- (0, 0.05);
\node[below] at (0, 0) {\small{$x$}};

\end{tikzpicture}
\caption{}
\end{subfigure}
\hfill
\begin{subfigure}{0.59\textwidth}
\centering
\begin{tikzpicture}[scale=0.8]

\draw[<->](-3, 0) -- (5, 0);
\draw[->](-2.5, 0) -- (-2.5, 4);
\draw[dotted] (-2.5, 2.5) -- (2, 2.5);
\draw[dotted] (-2.5, 3.5) -- (2.5, 3.5);

\draw (-1, -0.05) -- (-1, 0.05);
\draw (1.5, -0.05) -- (1.5, 0.05);
\draw (4, -0.05) -- (4, 0.05);
\draw (-2.55, 3.5) -- (-2.45, 3.5);
\draw (-2.55, 2.5) -- (-2.45, 2.5);
\draw[line width=2pt, blue!40!gray] (-1,0) -- (1.5,0);
\draw[line width=2pt, red!40!gray] (1.5,0) -- (4,0);

\node[left] at (-2.5 , 2.5) {\small{$T$}};
\node[left] at (-2.5 , 3.5) {\small{$\tau(\beta+1)$}};

\draw (0, -0.05) -- (0, 0.05);
\node[below] at (0, 0) {\small{$x$}};

\draw (-1, -0.05) -- (-1, 0.05);
\node[below] at (-1, 0) {\small{$i\ell$}};
\draw (1.5, -0.05) -- (1.5, 0.05);
\node[below] at (1.5, 0) {\small{$(i+1)\ell$}};
\draw (4, -0.05) -- (4, 0.05);
\node[below] at (4, 0) {\small{$(i+2)\ell$}};

\draw (0,0) to [out=90, in=200] (1,1) to [out=20, in=270] (2, 2.5);
\draw (1,0) to [out=90, in=200] (1,1.7) to [out=20, in=270] (3, 3) to [out=90, in=350] (2.5, 3.5);

\end{tikzpicture}
\caption{}
\end{subfigure}
\caption{The two conditions for an acceptable cell (in $d=1$). In (a), a representation of the first requirement can be seen: for every $x$ in the blue box that is occupied by time $\tau\beta$, there exists a path $\gamma^{x}$ with the aforementioned properties. On the left, the second requirement: regardless from where the infection starts in the blue box, the infection spreads to the top of nearby boxes by time $\tau(\beta+1)$.}
\label{fig:acceptable}
\end{figure}
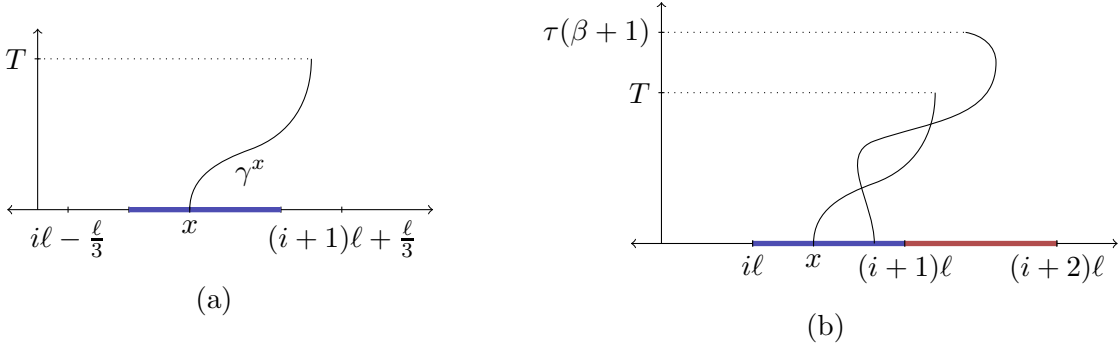

\newconstant{c:infection_spread}

The probability for the above event, but without recovery marks (thus, for the case $\lambda=0$), was derived in~\cite[Lemma 4]{gs1}. They prove that there exists a positive constant $C$ such that
\begin{equation}
\P_{\frac{\rho}{2}}^{0}[(i,\tau) \text{ is acceptable}] \geq 1-e^{-C\rho \ell^{\frac{1}{3}}}.
\end{equation}
We will use their result with a thinning argument where only particles that do not have recovery marks between times $\tau \beta$ and $(\tau+1)\beta$ are considered. 
\begin{lemma}\label{lemma:trigger_lipschitz_surface_1}
   Fix the ratio $\beta/\ell^{2}$. There exists a positive constant $\useconstant{c:infection_spread}>0$ such that, for all positive values of $\rho$ and $\lambda$ and all $\ell$ sufficiently large,
   \begin{equation}
   \P_{\frac{\rho}{2}}^{\lambda}[(i,\tau) \text{ is acceptable}] \geq 1-e^{-\useconstant{c:infection_spread}\rho e^{-\lambda \beta} \ell^{\frac{1}{3}}}.
   \end{equation}
\end{lemma}
\begin{proof} 
   With probability at least $e^{-\lambda \ell^{\frac{5}{3}}}$ the path $\gamma^x$ has no recovery marks up to time $T=\ell^{\frac{5}{3}}$ for each $x$, whereas a particle has no recovery marks during the interval $[\tau\beta,(\tau+1)\beta]$ with probability $e^{-\lambda \beta}$.
   Using the above and adjusting the constant in~\cite[Lemma 4]{gs1}, we obtain the lower bound $1-e^{-2C\rho e^{-\lambda \beta} \ell^{\frac{1}{3}}}-\ell^{d}e^{-\lambda \ell^{\frac{5}{3}}}$ from which the lemma follows.
\end{proof}

If we obtain a Lipschitz surface composed of boxes that are acceptable, and we know that the infection enters ``from the time dimension'' some box in the Lipschitz surface\footnote{The infection enters a cell $(i,\tau)$ 
from the time dimension if there exists an infected particle in the box $i$ at the initial time of the cell (that is, at time $\tau\beta$).}, 
we can apply Theorem~\ref{t:lipschitz} to conclude that we have local survival for our infection. This is however not a complete proof yet, 
since it may be the case that the infection crosses the Lipschitz surface without touching the bottom part of a cell. In order to take care of this, 
we now define \emph{good cells}, that will take into account this possibility.

For each cell $(i, \tau)$, 
fix an independent realization of a random walk path $(\gamma_{t}^{(i, \tau)})_{t \in [0, \beta]}$, with $\gamma^{(i, \tau)}_{0}=0$. 
We will change slightly the construction of the process so that these paths are used by infected particles when they enter the box through the spatial dimension. 
We say that the cell $(i,\tau)$ is \emph{good} if all the following events hold in  the super-cell $(i,\tau)$:
\begin{enumerate}
   \item the path $\gamma^{(i,\tau)}$ does not have any recovery marks between times $0$ and $\beta$ and, for each interval of time $[t, t+T] \subset [0, \beta)$, where $T= \ell^{\frac{5}{3}}$ and $t \leq \beta-T$, 
   the path $\gamma^{(i,\tau)}|_{[t,t+T]}$ does not have a displacement larger than $\ell/4^{d}$;
   
   \item for every $(x,t)$ in the cell $(i,\tau)$ such that $t-\tau\beta$ is a jump time of $\gamma^{(i,\tau)}$ and $t < (\tau+1)\beta-T$, there exists at least one $i' \in \Z^{d}$ with $||i'||_{\infty} \leq 1$ 
   for which the following holds: there is a particle that is inside the quasi-super box $i$ at time $\tau\beta$, has no recovery marks and is inside the super cell $(i,\tau)$ up to time $(\tau+1)\beta$, 
   intersects the path $\big(x+\gamma^{(i,\tau)}_{s-\tau \beta}-\gamma^{(i,\tau)}_{t-\tau \beta}\big)|_{s \in [t,t+T]}$ by time $t+T$, and is at time $(\tau+1)\beta$ inside the cell $(i+i', \tau+1)$;
   instead, when $t-\tau\beta$ is a jump time of $\gamma^{(i,\tau)}$ but $(\tau+1)\beta-T \leq t \leq (\tau+1)\beta$, 
   we require that $x+\gamma^{(i,\tau)}_{\beta}-\gamma^{(i,\tau)}_{t-\tau \beta}$ belongs to a box $i+i'$ with $\|i'\|_\infty \leq 1$;
   
   \item all cells $(i+i', \tau+1)$, for $||i'||_{\infty} \leq 1$, are acceptable.
\end{enumerate}

Although the event that a cell is good has a larger support than the event where this cell is acceptable (specially because of the third item in the definition), it is still the case that being good is a finitely supported condition. Figure~\ref{fig:good_box} contains a representation of the second requirement in the definition of a good cell.

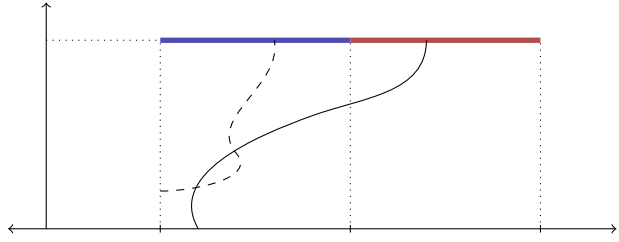
\begin{figure}
\centering
\begin{tikzpicture}

\draw[<->](-3, 0) -- (5, 0);
\draw[->](-2.5, 0) -- (-2.5, 3);
\draw[dotted] (-1, 0) -- (-1, 2.5);
\draw[dotted] (1.5, 0) -- (1.5, 2.5);
\draw[dotted] (4, 0) -- (4, 2.5);
\draw[dotted] (-2.5, 2.5) -- (2, 2.5);
\draw[line width=2pt, blue!40!gray] (-1, 2.5) -- (1.5, 2.5);
\draw[line width=2pt, red!40!gray] (1.5, 2.5) -- (4, 2.5);

\draw (-1, -0.05) -- (-1, 0.05);
\draw (1.5, -0.05) -- (1.5, 0.05);
\draw (4, -0.05) -- (4, 0.05);

\draw[dashed] (-1, 0.5) to [out=0, in=310] (0, 1) to [out=130, in=280] (0.5, 2.5);
\draw (-0.5, 0) to [out=120, in=200] (1, 1.5) to [out=20, in=270] (2.5, 2.5);

\end{tikzpicture}

\caption{An illustration of the second condition in the definition of good cell. The dashed line represents part of the path that is associated to the cell $(i, \tau)$.}
\label{fig:good_box}
\end{figure}

\newconstant{c:infection_spread_2}

\begin{lemma}\label{lemma:trigger_lipschitz_surface_2}
Fix the ratio $\beta/\ell^{2}$ small enough. There exists a constant $\useconstant{c:infection_spread_2}>0$ such that, 
for all positive values of $\rho$ and $\lambda$, and all $\ell$ sufficiently large,
\begin{equation}
\P_{\frac{\rho}{2}}^{\lambda}[(i,\tau) \text{ is good}] \geq 1-(3\ell^{d}\beta+2^{d}+1) e^{-\useconstant{c:infection_spread_2} \min \{1, \rho\} e^{-\lambda \beta} \ell^{\frac{1}{6}}}-2e^{-\min \{ \lambda, 1 \} \beta}.
\end{equation}
\end{lemma}

\begin{proof}
   Without loss of generality, we assume that $(i,\tau)=(0,0)$. We first observe that
   \begin{equation}
   \P[\gamma^{(0,0)} \text{ jumps more than } 3 \beta \text{ times}] \leq e^{-\beta}.
   \end{equation}
   Furthermore, if the number of jumps of $\gamma^{(0,0)}$ is bounded by $3\beta$, the probability that there exists some time $t$ such that $\gamma^{(0,0)}$ travels more than distance $\frac{\ell}{4^{d}}$ in the first $T$ units of time after time $t$ is bounded by $3\beta e^{-c\ell^\frac{1}{6}}$, since one can split this probability by examining only what happens on the times of each of the jumps and use the bound
\begin{equation}
\P[|X_{s}| \geq \alpha, \text{ for some } t \leq T] \leq e^{-\tilde{c}\frac{\alpha}{T^{\sfrac{1}{2}}}},
\end{equation}
valid for simple random walks. Besides, the probability that $\gamma^{(0,0)}$ does not have any recovery mark is exactly $e^{-\lambda \beta}$. From now on we assume that $\gamma^{(0,0)}$ satisfies the first condition of the definition and jumps at most $3\beta$ times.
   
   Suppose now that $(x,t)$ is some point in the cell $(0,0)$ such that $t \leq \beta-T$ is a jump time of $\gamma^{(0,0)}$, and consider the path $\gamma^{(0,0)}$ starting from $x$ at time $t$. We now need to bound the probability that there exists at least one infected particle that enters a cell $(i,1)$, with $||i||_{\infty} \leq 1$. Consider the ``translated'' extended box $x+\big[-\frac{\ell}{3}, \ell+\frac{\ell}{3} \big]^{d}$,  and observe that the density of particles contained inside this extended box that did not leave the super cell $(0,0)$ during the time interval $[0,t]$ can be lower bounded by $\frac{1}{4}\rho$ via a thinning argument. In fact, by taking $\beta/\ell^{2}$ small enough, we can ensure that particles that are inside the extended box $x+\big[-\frac{\ell}{3}, \ell+\frac{\ell}{3} \big]^{d}$  at time $t$ have probability of leaving the super cell $(0,0)$ before time $t$ bounded by $\frac{1}{2}$.
   
  Assume now that the density of particles in the extended box $x+\big[-\frac{\ell}{3}, \ell+\frac{\ell}{3} \big]^{d}$ at time $t$ is $\frac{1}{4}\rho$.
  To conclude, we need to bound the probability that the path $\gamma^{(0,0)}$ starting from $x$ at time $t$ intersects at least one particle that is in the extended box $x+\big[-\frac{\ell}{3}, \ell+\frac{\ell}{3} \big]^{d}$ at time $t$ that has no recovery times up to time $\beta$ and is at a cell $(i,1)$, with $||i||_{\infty} \leq 1$, at time $\beta$.
  In order to do so, one follows the same argument as in~\cite[Lemma 4]{gs1}, but restricted to the extended box $x+\big[-\frac{\ell}{3}, \ell+\frac{\ell}{3} \big]^{d}$, the time interval $[t, \beta]$ and a configuration inside this box at time $t$ with density $\frac{1}{4}\rho$.
  In other words, one needs to show that the number of particles that touch the path $\gamma^{(0,0)}$ between times $t$ and $t+T$ is large, by splitting time into smaller subintervals. 
  After this is done, it remains to deduce that at least one of these particles has the desired properties. This can be done by noticing that the number of such particles stochastically dominates a binomial random variable with appropriate parameters.
  We refrain from pursuing the fine details of the proof, but point out that, via union bound on the possible choices for $(x,t)$, we obtain the bound
   \begin{equation}
   \P_{\frac{\rho}{2}}^{\lambda}\Big[ \begin{array}{c} \text{the box $(0,0)$ does not satisfy} \\ \text{ condition 1 or 2 of a good box} \end{array} \Big] \leq 3\beta\ell^{d}e^{-\useconstant{c:infection_spread}\frac{\rho}{2} e^{-\lambda \beta} \ell^{\frac{1}{3}}} +3\beta e^{-c\ell^\frac{1}{6}}+e^{-\beta}+e^{-\lambda \beta}.
   \end{equation}
  In the equation above, the term $3\beta \ell^{d}$ comes from the union bound accounting for the possibilities of $(x,t)$. The probability multiplying $3 \beta \ell^{d}$ is an immediate application of~\cite[Lemma 4]{gs1} with the appropriate values for the density and sizes of the boxes. The final three terms in the probability above regard the restrictions we imposed in the path $\gamma^{(0,0)}$.
   
   Finally, the probability that all cells $(i,1)$, with $||i||_{\infty} \leq 1$ are acceptable can be bounded using Lemma~\ref{lemma:trigger_lipschitz_surface_1}. The proof is completed by combining all these estimates and simplifying the final expression.
\end{proof}

  The lemmas above are central in the proof of Theorem~\ref{t:local_survival}, but we first need to introduce an alternative construction of the process that will use the additional paths available in each space-time cell (recall these paths were introduced in the definition of good cells).
  This is necessary in order to assure that, when the infection process enters a cell, it spreads to neighboring cells but does not travel very far away.
  We begin by considering the same paths $(S_{t}^{x,n})_{t \geq 0}$ that we used in the first construction of the process, and evolve the infection process with it.
  The difference is that, depending on how the infected particles travel across space-time cells, they might change their paths to the additional one during its trajectory inside a given cell.

  Let us fix a cell $(i,\tau)$ and observe the process inside it.
  If there are infected particles at time $\tau\beta$ inside the box $i$, the path $\gamma^{(i,\tau)}$ is not used.
  Assume this is not the case, i.e., that all particles in the box $i$ are healthy at time $\tau\beta$.
  Observe now the process on the neighboring boxes $i+i'$, with $||i'||_{\infty} \leq 1$, and fix the first infected particle that leaves the extended box of $i+i'$ and enters the box $i$, if it exists.
  Associate to this selected first particle the path $\gamma^{(i, \tau)}$.
  The distinguished particle now follows this path until time $(\tau+1) \beta$, that is, if this first infected particle that enters the cell $(i, \tau)$ does so at time $s > \tau \beta$, it follows the trajectory given by $\gamma^{(i, \tau)}_{t-\tau \beta}$ for all time $t$ in the interval $(s, (\tau+1)\beta)$.
  Notice however that this particle may not use the path $\gamma^{(i,\tau)}$ up to time $(\tau+1)\beta$ if it happens that it leaves the super cell $(i,\tau)$ and enters a space-time box that has not yet seen an infected particle as described above.
  In this case, this particle changes the followed path to the one associated to the next box it enters.
  Since the tessellation into space-time cells is fixed, this new construction preserves the distribution of the process.
  
  We now proceed with the proof of Theorem~\ref{t:local_survival}. Notice that it suffices to conclude that~\eqref{eq:local_survival} holds for the graphical construction introduced above.
\begin{proof}[Proof of Theorem~\ref{t:local_survival}]
   We begin by setting the stage in order to apply Theorem~\ref{t:lipschitz}. Consider the event $E$ restricted to the super cell $(i,\tau)$ defined by
   \begin{equation}
   E(i, \tau) = \big\{\text{all cells } (i+i',\tau+\tau'), \text{ with } ||(i',\tau')||_{\infty} \leq 1, \text{ are acceptable and good} \big\}.
   \end{equation}
Notice that, straightforward from the definition of good and acceptable cells, we obtain that the event $E(i, \tau)$ is restricted to the super cell $(i,\tau)$ with $\eta= 4d+4$.
   If $(i,\tau) = (0,0)$, we write simply $E=E(0,0)$. Set $\epsilon = \frac{1}{2}$ and fix the ratio $\beta/\ell^{2}$ small enough in a way that~\eqref{eq:condition_omega} is verified for $\omega = 2\eta+1 = 8d+9$. Assume furthermore that $\ell$ is large enough so that Lemmas~\ref{lemma:trigger_lipschitz_surface_1} and~\ref{lemma:trigger_lipschitz_surface_2} hold. By possibly increasing the value of $\ell$ and choosing $\lambda$ sufficiently small (depending on $\ell$), we obtain
   $
   \P^{\lambda}_{\frac{\rho}{2}}[E^{c}] \leq e^{-\alpha_{0}},
   $
   which implies
   \begin{equation}
      \log\Big(\frac{1}{1-\nu_{E}((1-\epsilon)\rho, Q_{(2\eta+1)\ell}, Q_{\omega \ell}, \beta)} \Big) 
      = \log\Big(\frac{1}{1-\P_{\frac{\rho}{2}}^{\lambda}[E]} \Big) \geq \alpha_{0}.
   \end{equation}
   Finally, if $\ell$ is further increased, we can ensure that $\epsilon \rho \ell^{d} = \frac{1}{2}\rho \ell^{d} \geq \alpha_{0}$. 
   This verifies the hypotheses from Theorem~\ref{t:lipschitz} and implies the existence of a Lipschitz surface $\mathcal{S}$ of cells $(i, \tau)$ for which the event $E(i, \tau)$ holds.

  The Lipschitz surface we obtain surrounds the origin almost surely (in the sense of Theorem~\ref{t:lipschitz}). Furthermore, there are infinitely many cells that belong to the Lipschitz surface with all spatial coordinates $0$.
  If any of these cells is visited by infected particles, then an infected particle comes within distance $\ell$ of the origin in that cell, and has a positive probability of visiting the origin before recovering.
  To conclude the theorem then, it suffices to verify that infected particles visit these cells infinitely many times.

  Assume that the infection survives forever, let $\mathcal{C}$ denote the portion of the Lipschitz surface $\mathcal{S}$ surrounding the origin.
  Denote by $\mathcal{C}^{\circ}$ the collection of extended cells that are contained in the bounded region limited by $\mathcal{C}$ and connected to $(0,0)$.
  Since the infection survives forever, it must exit the set $\mathcal{C}^{\circ}$.
  In particular, either there exists a cell $(i, \tau)$ that is either in $\mathcal{C}$ or neighbors $\mathcal{C}$ such that the infection enters through the base, or there exists one cell $(i,\tau)$ in $\mathcal{C}$ such that the distinguished path $\gamma^{(i,\tau)}$ is used.
  Our goal is to prove that the infection always enters some cell in the surface through its base, and thus can spread to neighboring boxes, since all cells in the surface are acceptable.

  We split the discussion into two cases.
  Assume first that the infection enters a cell $(i, \tau)$ in $\mathcal{C}$ or that neighbors $\mathcal{C}$ through the bottom.
  In this case, since the boxes in $\mathcal{C}$ and neighboring boxes are acceptable, the infection spreads to the cells $(i+i',\tau+1)$, for all $i'$ such that $||i'||_{\infty} \leq 2$.
  In particular, the infection enters a box of the Lipschitz surface through its bottom and we are done. 
  
  Consider now the case when the infection enters a box $(i,\tau)$ through its side and that the path $\gamma^{(i,\tau)}$ is used.
  Since $(i,\tau)$ and all its neighboring cells are also good, the infection (that enters through the side and then changes path) spreads to at least one cell $(i+i', \tau+1)$, with $||i'||_{\infty} \leq 1$, which is also acceptable.
  Here we use the fact that the path $\gamma^{(i,\tau)}$ has displacement at most $\frac{\ell}{4^{d}}$ between the time $t$ it enters the cell and $t+T$, and that the same happens with all the paths for the neighboring boxes, implying that the infected particle can only change paths between the paths of the cell $(i,\tau)$ or its neighbors.
  Notice that the cell $(i+i', \tau+1)$ is not necessarily part of the surface $\mathcal{S}$.
  However, since it is acceptable, it spreads the infection to the cells $(i+i'+i'', \tau+2)$, for all $i''$ such that $||i''||_{\infty} \leq d+2$.
  In particular, there exists a cell in the Lipschitz surface whose bottom part is reached by the infection.

By construction, the infection reaches every cell in the Lipschitz surface connected to this entry cell by a path with growing time coordinates.
  We now need to argue that, for any cell $(b,h)$ in $\mathcal{S}$, there exists a path of cells in the Lipschitz surface that starts in $(b,h)$, goes forward in time, and reaches infinitely many cells with spatial coordinates zero.
  This is clear in dimension one, since the Lipschitz surface intersects height-zero cells infinitely many times.
  For dimensions $d \geq 2$ it suffices to argue that we eventually reach a box with all but one base coordinate equal to $0$.
  From this, we reduce the analysis to the $d=1$ case.
   
   Fix then $(b,h) \in \mathcal{S}$. Denote by $b_{\text{s}}$ and $b_{\text{t}}$ the spatial and temporal coordinates of $b \in \Z^{d}$, respectively. Let $b_{\text{s}} =b_{0} \sim b_{1} \sim \dots \sim b_{n} = 0$ be a nearest-neighbor path connecting $b_{\text{s}}$ to $0$ in $\Z^{d-1}$. Consider the path $\bold{b}_{i} = (b_i, b_{\text{t}}+i)$, for $i \in \{0, \dots, n\}$. And notice that $||\bold{b}_{i}-\bold{b}_{i-1}||_{\infty} \leq 2$. If we follow the Lipschitz surface from $(b, h)$ through the path $\bold{b}_{i}$, we eventually reach a cell with base coordinates $(0,b_{\text{t}}+n) \in \Z^{d-1} \times \Z$. This then reduces the problem to dimension $d=1$ and concludes the proof.
\end{proof}

\section{Absence of phase transition for large densities}\label{sec:survival}
~
In this section we present the proof of Theorem~\ref{t:survival}. This will also be based on the Lipschitz surface, but we will consider a slightly different event in this case.

Here, our approach is simpler, but we need to introduce some additional notation. For $u \in \mathcal{U} = \{+e_{i}, -e_{i}: i \leq d\}$ and some pair $(x,k)$, denote by $N_{u}^{k}(x)$ the amount of particles that are at $x$ at time $k$ and, before time $k+1$, jump to the site $x+u$. Also, write $N_{0}^{k}(x)$ for the number of particles that remain in $x$ between times $k$ and $k+1$.

Assume that our process has initial density $\rho$. Via thinning we obtain that, for each fixed pair $(x,k)$, the variables $N_{u}^{k}(x)$, $u \in  \mathcal{U} \cup \{0\}$, are independent and
\begin{equation}\label{eq:distribution_N}
\begin{split}
N_{u}^{k}(x) & \sim \poisson\Big( \frac{1-e^{-1}}{2d}\rho \Big) \qquad \text{for } u \neq 0, \\
N_{0}^{k}(x) & \sim \poisson\big( e^{-1}\rho \big).
\end{split}
\end{equation}

We say that a space-time point $(x,k)$ is \emph{good} if
\begin{equation}
N_{u}^{k}(x) \geq 1, \text{ for all } u \in \mathcal{U}, \text{ and } N_{0}^{k}(x) \geq 2,
\end{equation}
and define the increasing event $\tilde{E}$ restricted to the super cell $(0,0)$ as
\begin{equation}
\tilde{E} = \Big\{\text{all points } (x,k) \in \Z^{d} \times \Z \text{ in the super cell } (0,0) \text{ are good} \Big\}.
\end{equation}

Notice that, on the event $\tilde{E}$, whenever the infection with instantaneous recovery enters the cell $(0,0)$, it spreads to the neighboring cells. In particular, if there exists a Lipschitz surface such that the event $\tilde{E}$ holds in each cell of it, then the infection survives forever once it enters the surface.
The next lemma will help us to apply Theorem~\ref{t:lipschitz}.
\begin{lemma}
For all choices of $\rho$, $\beta$, $\ell$, and $\eta$, we have
\begin{equation}
\P_{\rho}^{\infty}[E^{c}] \leq \big((2\eta+1)\ell\big)^{d}\eta \beta \Big( 1-\big(1-e^{-\rho e^{-1}} - \rho e^{-1}e^{-\rho e^{-1}} \big) \big(1-e^{-\frac{1-e^{-1}}{2d}\rho} \big)^{2d} \Big)
\end{equation}
\end{lemma}

\begin{proof}
Note that, due to the independence of the variables $N_{u}^{k}(x)$ and~\eqref{eq:distribution_N}, we have
\begin{equation}
\P_{\rho}^{\infty}[\, \text{the point } (0,0) \text{ is good} \,] = \big(1-e^{-\rho e^{-1}} - \rho e^{-1}e^{-\rho e^{-1}} \big) \big(1-e^{-\frac{1-e^{-1}}{2d}\rho} \big)^{2d}.
\end{equation}
We now apply union bound to obtain
\begin{equation}
\begin{split}
\P_{\rho}^{\infty}[E^{c}] & = \P_{\rho}^{\infty}\Big[\begin{array}{c}\text{there exists a point } (x,k) \\ \text{in the super cell } (0,0) \text{ that is not good} \end{array} \Big] \\
& \leq \big((2\eta+1)\ell\big)^{d}\eta \beta \big( 1-\P_{\rho}^{\infty}[\, \text{the point } (0,0) \text{ is good} \,] \big) \\
& = \big((2\eta+1)\ell\big)^{d}\eta \beta \Big( 1-\Big(1-e^{-\rho e^{-1}} - \rho e^{-1}e^{-\rho e^{-1}} \big) \big(1-e^{-\frac{1-e^{-1}}{2d}\rho} \big)^{2d} \Big),
\end{split}
\end{equation}
concluding the lemma.
\end{proof}

With this lemma in hand, we can conclude the proof of Theorem~\ref{t:survival}.
\begin{proof}[Proof of Theorem~\ref{t:survival}]
As in the proof of Theorem~\ref{t:local_survival}, we first frame our problem in a way as to apply Theorem~\ref{t:lipschitz}. Fix the ratio $\beta/\ell^{2}$ small enough. Choose $\eta=1$, $\epsilon=\frac{1}{2}$, and fix $\beta$ and $\ell$ such that $\beta > \ell$ and~\eqref{eq:condition_omega} is verified for $\omega=3$.

Choose now $\rho_{+}$ large enough such that, for all $\rho \geq \rho_{+}$,
\begin{equation}
3^{d} \ell^{d} \beta \Big( 1-\Big(1-e^{-\frac{\rho}{2} e^{-1}} - \frac{1}{2}\rho e^{-1}e^{-\frac{\rho}{2} e^{-1}} \Big) \big(1-e^{-\frac{1-e^{-1}}{4d}\rho} \big)^{2d} \Big) \leq e^{-\alpha_{0}},
\end{equation}
where $\alpha_{0}$ is given by Theorem~\ref{t:lipschitz}, and observe that this implies that
\begin{equation}
\begin{split}
\log\Big(\frac{1}{1-\nu_{\tilde{E}}((1-\epsilon)\rho, Q_{(2\eta+1)\ell}, Q_{\omega \ell}, \beta)} \Big) & = \log\Big(\frac{1}{1-\nu_{\tilde{E}}(\frac{1}{2}\rho, Q_{3\ell}, Q_{3\ell}, \beta)} \Big) \\
& = \log\Big(\frac{1}{1-\P_{\frac{\rho}{2}}^{\infty}[\tilde{E}]} \Big) \geq \alpha_{0}.
\end{split}
\end{equation}

By further increasing the value of $\rho_{+}$, we can assume that $\epsilon\rho \ell^{d} \geq \frac{1}{2}\rho_{+} \ell^{d} \geq \alpha_{0}$. The hypotheses of Theorem~\ref{t:lipschitz} are thus verified and we obtain the existence of a Lipschitz surface of base-height cells $(b,h)$ such that $\tilde{E}_{\text{bh}}(b,h)$ holds.

This surface surrounds the origin and implies that the infection process with instantaneous recovery has a positive probability of not dying out, for every $\rho \geq \rho_{+}$. This concludes the proof.
\end{proof}

\bibliographystyle{plain}
\bibliography{mybib}

\end{document}